\documentclass[11pt,a4paper]{article} 

\usepackage{amsmath,latexsym,amssymb}
\usepackage{tikz}
\usetikzlibrary{arrows,decorations.pathmorphing,backgrounds,positioning,fit,shapes,shapes.misc}
\usepackage{verbatimbox}

\newcommand{\FIXMA}[1]{\verb+#1+}
\usepackage{tabularx}
\newcommand{\termvalue}[1] {\lbrack\!\lbrack #1 \rbrack\!\rbrack}
\newcommand{\ter}[1] {\lbrack\!\lbrack #1 \rbrack\!\rbrack}

\newcommand{\lthen}{\mathbin{\rightarrow}}    

\newcommand\Cw{C_\omega}

\usepackage{euscript}
\usepackage{eufrak}

\usepackage{latexsym, color, amssymb, amsfonts,amsmath}
\usepackage{graphicx}
\usepackage{mathrsfs}
\usepackage{latexsym}
\usepackage{mathrsfs}
\usepackage{multicol}
\usepackage{enumerate}
\newenvironment{axioms}{\begin{description}}{\end{description}}
\newcommand\AxiomItem[2]{\item[\qquad #1] \quad $#2$}

\newenvironment{proof}{\noindent\bf Proof. \rm}{\hfill $\mbox{\boldmath{$ \square$}}$}

\newtheorem{coro}{\bf Corollary}[section]
\newtheorem{theo}[coro]{\bf Theorem}
\newtheorem{defi}[coro]{\bf Definition}
\newtheorem{lem}[coro]{\bf Lemma}
\newtheorem{rem}[coro]{\bf Remark}
\newtheorem{definition}[coro]{\bf Definition}
\newtheorem{lemma}[coro]{\bf Lemma}

\usepackage{verbatimbox}

\title{ \large \textbf{Leibniz's law and  paraconsistent models of ZFC}}

\author{Aldo Figallo-Orellano\footnote{E-mail: \texttt{aldofigallo@ime.usp.br}}\\ [2mm] 
{\small Institute of Mathematics and Statistics, University of São Paulo (USP), Brazil}\\
}

\begin{document}
\maketitle
\begin{abstract}
Full models for some non-classical  Set Theories are exhibited, where these theories are built over Nelson's logics {\bf N4}, {\bf N3} and da Costa's logics $\Cw$ and $C^2_\omega$. These models are constructed over Fidel semantics in which they are specific first-order structures in the sense of Model Theory. These structures are known  in the literature as $F$-structures and they are not algebras in the universal algebra sense. With these structures, it is possible to give Adequacy Theorems for each logic mentioned before,  for the propositional level, using the saturated structures as Fidel and Odintsov have previously  exhibited.

In this paper, we start by considering  Set Theories type ZF over both  Nelson's and da Costa's logics using the propositional axioms for first-order formulas and set-theoretic axioms of Zermelo-Fraenkel. Looking for models for these Theories,  we construct  $F$-structures valued models as adapting the machinery developed to Boolean and Heyting valued models cases. In this sense, we can mention that the first presentation of the $F$-structures valued models was given by A. Figallo-Orellano and J. Slagter.

As a by-product, we display how it is possible to present paraconsistent models for ZFC based on da Costa's logics treated in this paper.

\end{abstract}

\section{Introduction}

Paraconsistency is the study of logic systems having a negation $\neg$  which is not explosive; that is, formulas $\alpha$ and $\beta$ exist in the language of the logic such that $\beta$ is not derivable from the contradictory set $\{\alpha, \neg\alpha\}$.  These systems are typically called non-explosive systems (i.e. non-explosive with respect to the negation) and, in this note, we will work with this notion of paraconsistency which is the most accepted for logicians.

Going back to Nelson logic. a paraconsistent version of Nelson logic was first introduced in \cite{AlmNel}, where the authors observed that a weaker system obtained from Nelson logic (\cite{Nel}) by deleting the axiom schema $\alpha\to (\neg \alpha \to \beta)$ could be used to reason under inconsistency without incurring in a trivial logic. Semantics for this paraconsistent version of Nelson logic have been studied by Odintsov in \cite{Od1}, who calls the logic {\bf N4}, in terms of what is known in the literature as Fidel structures and twist-structures. In the case of Fidel's structures ($F$-structures) for {\bf N4}, Odintsov's presentation is an adaptation  of Fidel's $F$-structures given by himself for {\bf N3} in \cite{F2}.

On the other hand, the C-systems $C_n$ ($n<\omega$) and $C_\omega$ are among the best-known contributions of da Costa and his  collaborators.  In particular, for $C_n$ is possible to define a new operator $\circ$ where it plays a role of recovering the classical theorems. These logics are known as {\em Logics of Formal Inconsistency}, see for instance \cite{EFFG,FS1}. Both $C_n$ ($n<\omega$) and $C_\omega$ are non-algebrizable logics with Blok-Pigozzi method, see  \cite{OFP}. This was one of the reasons why da Costa's systems remained without semantics for several years.  In these setting, Fidel presented semantics for  both $C_n$ ($c<\omega$) and $C_\omega$ through introducing an algebraic-relational structures that would bear his name; this idea was in fact based on  a generalization of Lindenbaum-Tarski process, \cite{F1}. Recently, Fidel semantics for a first-order logic was presented in \cite{Figa22}; moreover, it was displayed the relation with da Costa and Alves's bi-valuation semantics through a Representation Theorem. Other paraconsistent logics which are not algebraizable were introduced recently by means of degree-preserving construction in the paper  \cite{EFFG,EFFG1,AFO3}.

On the other hand, since  both $C^2_\omega$, $\Cw$, {\bf N4} and {\bf N3}  can be axiomatized as an extension of the {\em intuitionistic positive calculus}; therefore,  we can present   $F$-structures for them based on Boolean algebras for $C^2_\omega$   and  {\em implicative lattices} for $\Cw$ , {\bf N4} and {\bf N3}  in order to have sound and complete  systems.

Going to the central affair of this paper, recall that Boolean-valued models were introduced with the intention to simplify the {\em Forcing} method  by Vop\u{e}nka. Afterwards,  Scott,  Solovay and  Vop\u{e}nka in 1965 proposed to present models for Zermelo-Fraenkel Set Theory (ZF); an excellent exposition of the theory can be found in the Bell's book, \cite{Bell}. This idea allowed obtaining   models for the Intuitionistic version of  Zermelo-Fraenkel Set Theory working through  Heyting-valued models, \cite{Bell1}. Meanwhile using  Boolean-valued models, it is possible to see that the {\em Axiom of  Choice} (AC) is validated over them; i.e., it is commonly said that Boolean-valued models are models of ZFC.  Recently, it was considered F-structure valued models as new models of ZF in \cite{FS0}; these models are also models of da Costa logic $C_\omega$ that we shall use in these notes.

In this paper, we start by considering  Set Theories type ZF over both   Nelson's and da Costa's logics using the propositional axioms for first-order formulas and set-theoretic axioms of Zermelo-Fraenkel. Looking for models for these Theories,  we construct  $F$-structures valued models as adapting the machinery developed to Boolean and Heyting case but now using the signature  $\{\to,\wedge, \vee, \neg, \bot\}$. where the negation $\neg$ is the primitive symbol of the propositional logics considered in this paper. Since the $F$-structures valued models are built on Boolean and Heyting algebras, we are able to recover the classical and intuitionistic negation depending on the necessity. In particular, we will interested in  having a crisp (or standard) {\em  empty set} through the axiom (Empty Sep) from ZF. Curiously, we are also able to prove the existence of a  {\em paraconsistent empty set} treating the negation of (Empty Sep) as paraconsistent one. The paraconsistent set theory will be noted as  $ZF_{\bf N4}$,  $ZF_{C^2_\omega}$ and $ZF_{C_\omega}$; and clearly, $ZF_{\bf N3}$ is not paraconsistent, as it will be exhibited.

With these new structures, we will build a relation $\ter{\cdot}$ over the set of sentences, where the interpretation of predicates $\in$ and $\approx$ will be considered as Boolean and Heyting cases. The definition for positive closed formulas with $\ter{\cdot}$ will be the standard and for sentences with negation will request that verified Leibniz's law; that is to say, $\termvalue{u\approx v } \leq \termvalue{\neg\phi(u)\to \neg \phi(v) }$ for every sentence $\phi$. In the case of negation-free sentences, the relation $\ter{\cdot}$ became in a mapping. Moreover, if we take a fixed election of $\ter{\cdot}$ for negative formulas we also have a mapping. In this setting, the paraconsistent behavior of $\ter{\cdot}$ as a mapping will be proved when the propositional logic is also paraconsistent as much as the interpretation is proved to be well-defined in all cases. At the same time, we will consider $\ter{\cdot}$ as a (functional) valuation in the sense on the propositional logics considered here, i.e., the definition of $\ter{\cdot}$ is consistent with propositional axioms as well as with Leibniz's law. As an immediate consequence of the interpretation for $ZF_{C^2_\omega}$ and $ZF_{C_\omega}$, we will have that Boolean-valued models are models for them. Clearly, these models are consistent one for $ZF_{C^2_\omega}$ and $ZF_{C_\omega}$. In addition, we are able to construct paraconsistent models for these Theories as we will see.

On the other hand, Bell proved that (AC) implies  Law of Excluded Middle (LEM) then (AC) does not hold on Intuitionistic Zermelo Fraenkel set theory where (AC) is very weak, \cite{Bell2,Bell1}. The same occurs on  $ZF_{\bf N4}$  ($ZF_{\bf N3}$ ) because of the logics {\bf N4} ({\bf N3}) does not hold LEM. In contrast, the systems $ZF_{C^2_\omega}$ and $ZF_{C_\omega}$ are built with LEM. So, we are able to construct paraconsistent models for ZFC using the models of $ZF_{C^2_\omega}$ and $ZF_{C_\omega}$.

\section{Preliminaries} 

 To present the axiomatization of the  logics studied in this paper,  we start firstly displaying  some notation useful for the rest of this manuscript. For a given logic $X$, let us consider the  signature $\Sigma=\{\wedge,\vee,\to,\neg\}$ for  $X$, and the language  $\mathfrak{Fm}$, or set of formulas, over the denumerable set of variable $Var$. In what follows we will consider several logics under the same signature $\Sigma$ and $X$ represents any of them. Recall now that the intuitionistic positive calculus (${\bf Int^+}$) may be  axiomatized by the following schemas and and the rule {\em Modus Ponens}:

\begin{axioms}
\AxiomItem{(Ax1)}{\alpha \lthen (\beta \lthen \alpha)},

\AxiomItem{(Ax2)}{\big(\alpha \lthen (\beta \lthen \gamma)\big) \lthen \big((\alpha \lthen \beta)
\lthen(\alpha \lthen \gamma)\big)},

\AxiomItem{(Ax3)}{(\alpha \land \beta) \lthen \alpha},

\AxiomItem{(Ax4)}{(\alpha \land \beta) \lthen \beta},

\AxiomItem{(Ax5)}{\alpha \lthen \big(\beta \rightarrow (\alpha \land \beta)\big)},
\AxiomItem{(Ax6)}{\alpha \lthen (\alpha \lor \beta)},

\AxiomItem{(Ax7)}{\beta \lthen (\alpha \lor \beta)},

\AxiomItem{(Ax8)}{(\alpha \lthen \gamma) \lthen \big((\beta \lthen \gamma) \lthen (\alpha \lor \beta
\lthen \gamma)\big)}.
\end{axioms}

Paraconsistent Nelson's logic, {\bf N4}, is defined as an extension of  ${\bf Int^+}$ adding the following schemas:
\noindent {\bf{Axioms}}

\begin{axioms}
\AxiomItem{(PN1)} {\neg(\neg\alpha)\leftrightarrow\,\alpha},
\AxiomItem{(PN2)} {\neg(\alpha\vee\beta)\leftrightarrow\,\neg\alpha\,\wedge\neg\beta},
 \AxiomItem{(PN3)}{\neg(\alpha\wedge\beta)\leftrightarrow\,\neg\alpha\,\vee\neg\beta},
   \AxiomItem{(PN4)}{\neg(\alpha\to \beta)\leftrightarrow \alpha\,\wedge\neg \beta}.
\end{axioms}

To define the logic  {\bf N3} , we only add the following  axiom schema to  {\bf N4}:

\begin{axioms}
\AxiomItem{(Ax9)} {(\alpha\wedge \neg \alpha) \to \beta}
\end{axioms}

It is clear that the formula {\bf (Ax9)'} $\alpha\to (\neg \alpha \to \beta)$ is equivalent to {\bf (Ax9)}; although, {\bf (Ax9)'}  is the axiom used initially to define the logic {\bf N3}, but using meta--deduction theorem, it is possible to see that they are equivalent. We need  {\bf (Ax9)} to build the $F$-structure for {\bf N3} as we can see in Fidel's paper \cite{F2}.

Let us consider the logic $\Cw$  is defined through adding the followings schemas to ${\bf Int^+}$:

\begin{axioms}

\AxiomItem{(C1)}{\alpha\lor \lnot \alpha},

\AxiomItem{(C2)}{\lnot\lnot \alpha\lthen \alpha},
\end{axioms}

Lastly, consider the logic $C^2_\omega$ is defined through adding the followings schemas to $\Cw$:

\begin{axioms}

\AxiomItem{(C3)}{((\alpha \to \beta)\to \alpha)\to \alpha},

\end{axioms}

It is worth mentioning that the axiom {\bf (C3)} can not be derived from the logic $\Cw$, see for instance \cite[page. 32]{F1}; moreover, the logics $C^2_\omega$ is not algebrizable with Blok-Pigozzi method and this logic is paraconsistent, \cite[Theorem 5.1 and Corollary 5.1]{OFP}. The notion  of derivation of a formula $\alpha$ on any logic $X$ is defined as usual.  We say that $\alpha$ is derivable from $\Gamma$ in $X$, denoted by $\Gamma\vdash\alpha$, if there exists a derivation of $\alpha$ from $\Gamma$ in $X$.  If $\Gamma=\emptyset$ we denote $\vdash\alpha$; in this case, we say that $\alpha$ is a theorem of $X$. 
 
\subsection{$F$-structures for {\bf N4}}\label{N4}

Next, we will make a synthesis of the $F$-structure  obtained in \cite{Od1} for {\bf N4} and to simplify reading, we summarize the fundamental concepts. Firstly, recall that an algebra   ${\bf A}=\langle A,\vee,\wedge,\to,0,1\rangle$ is said to be a {\em Heyting algebra} if the reduct $\langle A,\vee,\wedge,0,1\rangle$  is a bounded distributive lattice and for any $a,b\in A$ the value of $a\to b $  is a pseudo-complement of $a$ with respect to $b$; i.e.,  the greatest element of the set $\{z\in A: a\wedge c\leq b\}$. As a more general case, we say that ${\bf B}=\langle B,\vee,\wedge,\to,1\rangle$ is a {\em generalized Heyting algebra} if  $\langle B,\vee,\wedge,1\rangle$  is a distributive lattice with the greatest element $1$ and $\to$ is defined as Heyting case. This last class of algebras is called  {\em Relatively pseudo-complemented lattices} in \cite[Chapter IV]{RA} and {\em Implicative Lattices} in \cite{Od1}.

On the other hand, we also need to recall  that a $F$-structure for {\bf N4}  is a  system $\langle {\bf A},\{N_x\}_{x\in A}\rangle$ where $\bf A$ is a generalized Heyting algebra and $\{N_x\}_{x\in A}$ is a family of sets of $A$ such that the following conditions hold:
\begin{itemize}
  \item[\rm (i)] for any $x\in A$, $N_x\not=\emptyset$,
  \item[\rm (ii)] for any $x,y\in A$, $x'\in N_x$ and $y'\in N_y$, the following relations hold $x'\vee y'\in N_{x\wedge y}$ and $x'\wedge y'\in N_{x \vee y}$, $x\in N_{x'}$,
  \item[\rm (iii)] for any $x,y\in A$, $y'\in N_x$, we have $x\wedge y' \in N_{x\to y}$.
\end{itemize}

For the sake of simplicity, we will write sometimes $\langle {\bf A},N\rangle$ instead of $\langle {\bf A},\{N_x\}_{x\in A}\rangle$. Besides, we denote by  {\bf N4}-structures  the  $F$-structures for {\bf N4}. As an example of {\bf N4}-structure, we can take a generalized Heyting algebra $\bf A$ and the set $N_x^s=A$ for every $x\in A$. The structure $\langle A,\{N_x^s\}_{x\in A}\rangle$ will be said to be a saturated one;  these saturated structures are essential to prove the Adequacy Theorem. 
For a given set of formulas $\Gamma$ in {\bf N4}, we will consider the binary relation between formulas as follows:

\begin{center}
$\alpha\equiv_\Gamma\beta$ iff $\Gamma \vdash(\alpha\to\beta)\wedge(\beta\to\alpha).$ 
\end{center}

Bearing  in mind the positive axioms of {\bf N4}, (A1) to (A8), we have that $\equiv_\Gamma$ is a congruence with respect to the connectives  $\vee$, $\wedge$ and $\to$. With $|\alpha|_\Gamma$ we denote the class of $\alpha$ under $\equiv_\Gamma$ and $\mathcal{L}_{\bf N4}$ denotes the set of all classes. We can define the operations $\vee$, $\wedge$ and $\to$ on  $\mathcal{L}_{\bf N4}$  as follows: $|\alpha \# \beta|_\Gamma=|\alpha|_\Gamma\#|\beta|_\Gamma$ with $\#\in \{\wedge,\vee,\to \}$. Thus, it is clear that $\langle \mathcal{L}_{\bf N4},\wedge,\vee,\to,1\rangle$ is a generalized Heyting algebra with the greatest element $1=| \alpha \to \alpha|_\Gamma$. To extend the latter algebra to {\bf N4}-structure,  let us  define the set $N_{|\alpha|}$ for each formula $\alpha$ as follows: $N_{|\alpha|_\Gamma} = \{|\neg\beta|_\Gamma: \beta \equiv_\Gamma \alpha\}$. It is not hard to see that $\langle\mathcal{L}_{\bf N4}, \{N_{|\alpha|_\Gamma}\}_{\alpha\in \mathfrak{Fm}}\rangle$ is a {\bf N4}-structure that we will call Lindenbaum structure. 

Besides, we say that a function $v: \mathfrak{Fm}\to \langle {\bf A},\{N_x\}_{x\in A}\rangle$ is a  {\bf N4}-valuation if the following conditions hold for $\alpha$ and $\beta$ formulas:

\begin{itemize}
 \item[{\rm (v1)}] $v(\alpha) \in A$ where $\alpha$ is an  atomic formula,
 \item[\rm (v2)] $v(\neg\alpha)\in N_{v(\alpha)}$,
 \item[\rm (v3)] $v(\alpha\#\beta)=v(\alpha)\# v(\beta)$ where $\#\in\{\wedge,\vee,\to\}$,
\item[\rm (v5)] $v(\neg\neg \alpha)=v(\alpha)$,
 \item[\rm (v6)] $v(\neg(\alpha\vee \beta))=v(\neg \alpha)\wedge v(\neg \beta)$ and  $v(\neg(\alpha\wedge \beta))=v(\neg \alpha)\vee v(\neg \beta)$,
\item[\rm (v7)] $v(\neg(\alpha\to \beta))=v(\alpha)\wedge v(\neg \beta)$.
\end{itemize}

Since $v$ is a function, then $v(\neg \alpha)$ is unique and well-defined; clearly, we do not know which is the real truth-value taken by $v$. Moreover, it is not hard to see that if we take a {\bf N4}-valuation for atomic formulas, it works well on  formulas with more complexity as we can see in \cite{Od1}. For us,   a formula $\alpha$ will be {\em semantically valid} from $\Gamma$   if for every saturated {\bf N4}-structure and every valuation $v$ on the structure, the condition  $v(\beta)=1$ for every $\beta\in\Gamma$ implies $v(\alpha)=1$. In this case, we denote  $\Gamma\vDash_{\bf N4}\alpha$.

\begin{theo}{\rm \cite[Theorem 3.8]{Od1}}\label{teo4} Let $\Gamma \cup \{\alpha\}$ be a set of formulas of {\bf N4}. Then, $\Gamma\vdash\alpha$  if only if $\Gamma \vDash_{\bf N4} \alpha$
\end{theo}

In order to give the notion of  {\bf N3}-structure, we will consider the  {\bf N4}-structure $\langle {\bf A},N\rangle$ with the following condition:

\begin{itemize}
  \item[\rm (iv)] for any $x\in A$ and  $x' \in N_x$, then $x\wedge x'= 0$.
\end{itemize}

Furthermore, we can give the notion of {\bf N3}-valuation taking the function $v: \mathfrak{Fm}\to \langle {\bf A},\{N_x\}_{x\in A}\rangle$, we say that it  is a  {\bf N3}-valuation if $v$ is a {\bf N4}-valuation and the following conditions holds for a formula $\alpha$:
\begin{itemize}
  \item[\rm (v8)] $v(\alpha \wedge \neg \alpha)=0$.
\end{itemize}

So, we have that the following Theorem with a similar proof  to \cite[Theorem 3.8]{Od1}:

\begin{theo}  Let $\Gamma \cup \{\alpha\}$ be a set of formulas of {\bf N3}. Then, $\Gamma\vdash\alpha$  if only if $\Gamma \vDash_{\bf N3}\alpha$
\end{theo}

\subsection{$F$-structures  for $C_\omega$ and $C^2_\omega$ }\label{Cw} 

In the sequel, we will summarize the concept presented in the paper \cite{F1} by Fidel.  It is worth mentioning that in this paper,  a weak version of Adequacy Theorem was presented for some logics with respect to certain classes of $F$-structures. Start by recalling that a $F$-structure for $C_\omega$ ($C^2_\omega$), for short  $C_\omega$-structure ($C^2_\omega$-structure),  is a  system $\langle {\bf A},\{N_x\}_{x\in A} \rangle$ where $\bf A$ is a  generalized Heyting algebra (Tarski algebra with infimum)  and $\{N_x\}_{x\in A}$ is a family of  sets of $A$ such that the following conditions hold:
\begin{itemize}
  \item[\rm (i)] for very  $x'\in N_x$,   $x\vee x'=1$,
  \item[\rm (ii)] for every $x'\in N_x$ there is $x''\in N_{x'}$ such that $x''\leq x$.
\end{itemize}

Besides, we say that a function $v: \mathfrak{Fm}\to \langle {\bf A},\{N_x\}_{x\in A}\rangle$ is a  $C_\omega$-valuation if the following conditions hold:
\begin{itemize}
  \item[\rm (v1)] $v(\alpha)\in A$ where $\alpha$ is an  atomic formula,
  \item[\rm (v2)] $v(\alpha\#\beta)=v(\alpha)\# v(\beta)$ where $\#\in\{\wedge,\vee,\to\}$,
   \item[\rm (v3)] $ v(\neg\alpha)\in N_{v(\alpha)}$, $v(\neg\neg\alpha)\leq v(\alpha)$.
\end{itemize}

 As example of {$C_\omega$}-structure, we can take a generalized Heyting algebra $\bf A$ and the set $N_x^s=\{y\in A: x\vee y =1 \}$ for every $x\in A$. The structure $\langle A,\{N_x^s\}_{x\in A}\rangle$ will be said to be a saturated one; as {\bf N4} case,  these saturated structures are essential to prove the Adequacy Theorem. It is easy to see that taking the structure $\langle A,\{N_x^s\}_{x\in A}\rangle$, we have that $N^s_1=A$ and $1\in N^s_x$ for every $x\in A$.

A formula $\alpha$ is said to be {\em semantically consequence} for $\Gamma$ if  every  $C_\omega$-structure $\langle {\bf A}, N\rangle$ and every valuation $v$ on the structure, then condition $v(\alpha)=1$ holds; and in this case, we denote $\vDash_{\Cw}\alpha$. Moreover, we write  $\Gamma\vDash_{\Cw}\alpha$ if   every  $C_\omega$-structure $\langle {\bf A}, N \rangle$  and every valuation $v$ on the structure, the condition $v(\beta)=1$ for every $\beta\in\Gamma$ implies  $v(\alpha)=1$.

\begin{theo}{\rm \cite[Section 6]{OFP}}  Let $\Gamma \cup \{\alpha\}$ be a set of formulas of $C_\omega$. Then, $\Gamma\vdash\alpha$  if only if $\Gamma\vDash_{\Cw}\alpha$
\end{theo}

In a similar way, we can define $F$-structures for the logic $C^2_\omega$, but now the system $\langle {\bf A},\{N_x\}_{x\in A} \rangle$ is formed by a Tarski algebra with infimum $\bf A$. Recall that the class of Tarski algebras is $\{\to,1\}$-subreducts of the class of Boolean algebras, this variety is generated by $\langle {\bf 2}, \to, 1\rangle$ where ${\bf 2}=\{0,1\}$. Moreover, it is well-known that every Tarski algebra is a join-semilattice and so the language that can be considered for every  Tarski algebra with infimum is $\{\to,\vee, \wedge, 1\}$. Another important property of the theory is that every Tarski algebra with first element is a Boolean algebra, see for instance \cite{Figa16}. Clearly, we can take the same notion of  $C_\omega$-valuation but now for $C^2_\omega$ using a Tarski algebra with infimum over the algebraic support of the structure. So, we have:

\begin{theo}{\rm \cite[Section 6]{OFP}}  Let $\Gamma \cup \{\alpha\}$ be a set of formulas of $C^2_\omega$. Then, $\Gamma\vdash\alpha$  if only if $\Gamma\vDash_{C^2_\omega}\alpha$
\end{theo}
\subsection{Zermelo-Fraenkel Set Theory: Boolean and Heyting valued models} \label{BH}

In this part of the paper, we will present {\em Non-classical Set Theories} build over (paraconsistent) propositional logics, expressing its axioms on a suitable first-order version for Set Theory. We start by considering any logic $X$ over the first-order signature ${\Theta}_\omega$ which contains an equality predicate\, $\approx$ \, and a binary predicate $\in$, where this logic $X$ could be {\bf N4}, {\bf N3}, $\Cw$ or $C^2_\omega$ over a fixed signature $\{\to,\wedge,\vee, \neg, \bot\}$; that is to say, for every formula $\alpha$ we have $\sim \alpha:=\alpha \to \bot$. Clearly, these logics have a $\top$ given by $\top:=\alpha\to \alpha$. At the same time, we will consider also the language in terms of the signature $\{\to,\wedge,\vee, \neg, \bot\}$ for our Set Theories, where the negation $\neg$ is in the original language of mentioned (paraconsistent) propositional logics; and with the bottom, we are able to express the classical or intuitionistic negation depending on the requirements.
 The system $ZF_X$ is the first-order system obtained from the logic $X$ over $\Theta_\omega$ by adding the following set-theoretic axiom schemas:
 \begin{axioms}
\AxiomItem{(Extensionality)}{\forall x \forall y [\forall z(z\in x \leftrightarrow z\in y)\to (z \approx y)]},\AxiomItem{(Pairing)}{\forall x \forall y \exists w\forall z [z\in w \leftrightarrow (z \approx x \vee z \approx y) ]},\AxiomItem{(Collection)}{\forall x [(\forall y\in x\exists z\phi(y,z))\to \exists w\forall y\in x \exists z\in w \phi(y,z)]},
where $w$ is not free in $\phi(y,z)$,
\AxiomItem{(Powerset)} {\forall x \exists w \forall z [z\in w \leftrightarrow \forall y\in z(y\in x)]},\AxiomItem{(Separation)} {\forall x \exists w \forall z[z\in w \leftrightarrow (z\in x \wedge \phi(z))]},
where $w$ is not free in $\phi(z)$,
\AxiomItem{(Empty set)} {\exists x\forall z[z\in x \leftrightarrow ((z\approx z)\to \bot)]},
The set satisfying this axiom is, by extensionality, unique and we refer to it with notation $\emptyset$.
\AxiomItem{(Union)}{\forall x \exists w \forall z[ z\in w \leftrightarrow \exists y\in x(z\in y)]},
\AxiomItem{(Infinity)}{\exists x [\emptyset\in x \wedge \forall y\in x (y^+\in x)]},
From union and pairing and extensionality, we can note by $y^+$ the unique set $y\cup \{y\}$. 
\AxiomItem{(Regularity)}{\forall x [( \forall y\in x \phi(y))\to \phi(x)]\to \forall x\phi(x)}.
where $y$ is not free in the formula $\phi(x)$\end{axioms}
These axioms are typically used to define Zermelo-Fraenkel Set Theory (for short {\bf ZF}), however they are expressing expressing the classical and intuitionistic negation mentioned before for the axiom (Empty Set), see \cite{Bell}. In this paper, we will focus on the task of presenting models for $ZF_{\bf N4}$, $ZF_{\bf N3}$, $ZF_{C^2_\omega}$ and $ZF_{C_\omega}$, where these theories are constructed over the logic $X$ mentioned before. It is important to note that both the logic {\bf N4}, $\Cw$ and $C^2_\omega$ have no bottom in the original language but one can extend the language by the symbol $\bot$ with the added axiom $\alpha\to \bot$. In the sequel, $Fun(f)$ means that $f$ is a relation, i.e. a set of ordered pairs, such that $f$ is a function, while $dom(f)$ and $rang(f)$ will denote the domain and range of a given function $f$.
\begin{axioms}
\AxiomItem{(Axiom of choice) (AC)}{ \forall u\exists f [Fun(f)\wedge dom(f)\approx u\wedge \forall x\in u[((x\approx \emptyset)\to \bot) \to f(x)\in x] ]}.\end{axioms}


We will now consider  the class Boolean algebras and Heyting algebras. Let us suppose $\bf A$ is a Boolean or Heyting algebra, and consider  ${\mathbf{V}}^{\bf A}$  of Boolean or Heyting valued model over  ${\bf A}$ as follows:  let us construct a universe of {\em names}, and let $\xi$ be an ordinal number, by transfinite recursion:
 
$$ {\mathbf{V}_\xi}^{A}=\{x: x\, \textrm{\rm  a function and }\, ran(x)\subseteq A \, \,\textrm{\rm and }\,  dom(x)\subseteq \mathbf{V}_\zeta^{A}\, \textrm{\rm for some}\,  \zeta< \xi \} $$

$$ {\mathbf{V}}^{A}=\{x: x\in {\mathbf{V}_\xi}^{A}\, \textrm{\rm  for some } \xi \} $$

The class ${\mathbf{V}}^{A}$ is called the algebraic-valued model over $A$. Now, by ${\cal L}_\in$ we denote the first-order language of set theory which consists of the propositional connectives $\{\to, \wedge, \vee, \neg\}$ of the classical or intuitionistic logics and two binary predicates $\in$ and $\approx$; in this case, it is not necessary to consider $\bot$ and $\top$ in the signature.  We can also expand this language by adding all the elements of ${\mathbf{V}}^{\bf A}$; the expanded language we will denote ${\cal L}_{\bf A}$. Now, we will  define a valuation by induction on the complexity of a closed formula in ${\cal L}_{\bf A}$. Then,  the mapping $\ter{\cdot}^{\bf A}:{\cal L}_{\bf A}\to {\bf A}$ is defined as follows:

{\small
\begin{align}
\termvalue{u\in v}^{\langle {\bf A}, N\rangle} & =\bigvee\limits_{x\in dom(v)} (v(x) \wedge \termvalue{x \approx u}^{\langle {\bf A}, N\rangle});\tag{BH-1}\\[3mm]
\ter{u\approx v}^{\langle {\bf A}, N\rangle} &=\bigwedge\limits_{x\in dom(u)} (u(x) \to \ter{x\in v}^{\langle {\bf A}, N\rangle}) \wedge \bigwedge\limits_{x\in dom(v)} (v(x) \to \ter{x\in u}^{\langle {\bf A}, N\rangle});\tag{BH-2}\\[3mm]
\ter{\varphi \# \psi}^{\langle {\bf A}, N\rangle}  &= \ter{\varphi}^{\langle {\bf A}, N\rangle}  \# \ter{\psi}^{\langle {\bf A}, N\rangle}\,\, \text{for every}\,\, \#\in \{\wedge,\vee, \to\};\tag{BH-3}\\[3mm]
\ter{\neg \varphi}^{\langle {\bf A}, N\rangle}& =\neg \ter{\varphi}^{\bf A};\tag{BH-4}\\[3mm]
\ter{\exists x\varphi}^{\langle {\bf A}, N\rangle} &= \bigvee\limits_{{u\in {\mathbf{V}}^{\langle A, N\rangle}}} \ter{\varphi (u)}^{\langle {\bf A}, N\rangle}\,\, \text{and}\,\, \ter{\forall x\varphi}^{\langle {\bf A}, N\rangle} = \bigwedge\limits_{{u\in {\mathbf{V}}^{\langle A, N\rangle}}} \ter{\varphi (u)}^{\langle {\bf A}, N\rangle};\tag{BH-5}
\end{align}
}

A sentence $\varphi$ in the language ${\cal L}_{A}$  is said to be valid in ${\mathbf{V}}^{ A}$, which  is denoted by ${\mathbf{V}}^{\bf A} \vDash \varphi$, if $\termvalue{\varphi}^{\bf A}=1$.  With this interpretation $\ter{\cdot}^{\bf A}$, it is possible to prove the validity of all set-theoretic axiom schemas of ZF. Furthermore, it is possible to prove the validity of (AC)  with Boolean-valued models. 

On the other hand, Bell proved that (AC) implies  Law of Excluded Middle (LEM) then (AC) does not hold on Intuitionistic Zermelo Fraenkel set theory where (AC) is very weak, \cite{Bell2,Bell1}. The same occurs on  $ZF_{\bf N4}$  ($ZF_{\bf N3}$ ) because {\bf N4} ({\bf N3}) does not hold LEM. In contrast, the systems $ZF_{C^2_\omega}$ and $ZF_{C_\omega}$ are built with LEM as an axiom.

\section{Some properties of Leibniz's law}\label{section4}

In this section, we shall analyze the minimal conditions for models constructed over a Heyting (or Boolean) algebra that we need to prove that some Zermelo-Fraenkel's set-theoretic axioms are valid in these models. Specifically, we will prove the validity of the axioms (Separation), (Pairing), (Union) and (Collection). 
Now, we fix a model of set theory $\mathbf{V}$ and a completed Heyting algebra $A$. Let us construct a universe of {\em names} by transfinite recursion: $$ {\mathbf{V}_\xi}^{\bf A}=\{x: x\, \textrm{\rm a function and }\, ran(x)\subseteq A \, \,\textrm{\rm and }\, dom(x)\subseteq \mathbf{V}_\zeta^{A}\, \textrm{\rm for some}\, \zeta< \xi \} $$
$$ {\mathbf{V}}^{\bf A}=\{x: x\in {\mathbf{V}_\xi}^{A}\, \textrm{\rm for some } \xi \} $$

The class ${\mathbf{V}}^{\bf A}$ is called the algebraic-valued model over $A$. Let us observe that we only need the set $A$ in order to define ${\mathbf{V}_\xi}^{\bf A}$. By ${\cal L}_\in$, we denote the first-order language of set theory which consists of only the propositional connectives $\{\to, \wedge,\vee \}$ and two binary predicates $\in$ and $\approx$. We can expand this language by adding all the elements of ${\mathbf{V}}^{\bf A}$; the expanded language we will denote ${\cal L}_{\bf A}$. For this construction of models, we also have {\bf Induction principles} as the Boolean and Intuitionistic case, \cite{Bell}.
 Now, for a given completed Heyting (Boole) algebra $\bf A$, the following relation $\ter{\cdot}^{\bf A}\subseteq {\cal L}_{\bf A}\times A$ is defined as follows:

\begin{align}\termvalue{u\in v }^{\bf A} & =\bigvee\limits_{x\in dom(v)} (v(x) \wedge \termvalue{x \approx u }^{A});\tag{HV1}\\[3mm]\termvalue{u\approx u }^{\bf A}&=1;\tag{HV2}\\[3mm]\termvalue{u\approx v }^{\bf A}& \leq \termvalue{\phi(u)\to \phi(v) }^{\bf A}\,\, \text{for all formula}\,\, \phi(x); \tag{HV3}\\[3mm]\termvalue{\varphi \# \psi}^{\bf A} &= \termvalue{\varphi}^{\bf A} \# \termvalue{\psi} ^{\bf A}\,\, \text{for every}\,\, \#\in \{\wedge,\vee, \to\};\tag{HV4}\\[3mm]\termvalue{\exists x\varphi}^{\bf A} & = \bigvee\limits_{{u\in \mathbf{V}}^{\bf A}} \termvalue{\varphi (u)}^{\bf A}\,\,\text{and} \,\,\termvalue{\forall x\varphi}^{\bf A} = \bigwedge\limits_{{u\in \mathbf{V}}^{\bf A}} \termvalue{\varphi (u)}^{\bf A}.\tag{HV5}\end{align}
\noindent In the rest of Section, we will consider a fixed but arbitrary mapping $\termvalue{\cdot}^{\bf A}$; i.e., the relation is taken as a function. In this sense, $\termvalue{\varphi}^{\bf A}$ is called the possible {\bf truth-values} of the sentence $\varphi$ in the language ${\cal L}_{ A}$ in the algebraic-valued model over $A$. Besides, the condition (HV3) is known as {\em Leibniz's law} and it works as an axiom. In principle, we do not know the exact value of $\termvalue{u\approx v }^{\bf A}$ but clearly it exists. The truth-value of $\termvalue{u\approx v }^{\bf A}$ can be taken from the standard form given by (BH-2) of Section \ref{BH} or another defining by $\termvalue{u\approx v }^{\bf A}=1$ iff $u=v$ and $\termvalue{u\approx v }^{\bf A}=0$ iff $u\not=v$ for $u,v\in {\mathbf{V}}^{\bf A}$. So, we have various possibilities for defining $\termvalue{\varphi}^{\bf A}$; for all of them, we will represent them with the same notation in order to simplify the reading. For each of these possibilities, we will have a unique mapping based on the fixed value of $\termvalue{u\approx v }^{\bf A}$.

\begin{defi}
A sentence $\varphi$ in the language ${\cal L}_{A}$  is said to be valid in ${\mathbf{V}}^{ \bf A}$, which  is denoted by ${\mathbf{V}}^{\bf A} \vDash \varphi$, if $\termvalue{\varphi}^{\bf A}=1$.

\end{defi}

\begin{lemma}\label{lema1}
For a given complete Heyting algebra $\bf A$. Then,  $u,v\in {\mathbf{V}}^{\bf A}$ we have

\begin{itemize}

\item[\rm (i)] $\termvalue{u\approx v}^{\bf A}= \termvalue{v\approx u}^{\bf A}$,
\item[\rm (ii)] $u(x)\leq \termvalue{x\in u}^{\bf A}$ for every $x\in dom(u)$.

\end{itemize}
\end{lemma}

\begin{proof}
(i) Let us consider the formula $\phi(z):= u \approx z$. Then, $\termvalue{u\approx v}^{\bf A} \leq \termvalue{\phi(u)\to \phi(v)}^{\bf A}=\termvalue{u\approx u}^{\bf A}\to \termvalue{v\approx u}^{\bf A}=1\to \termvalue{v\approx u}^{\bf A}=  \termvalue{ v\approx u}^{\bf A}$. Analogously, we have $\termvalue{v\approx u}^{\bf A}\leq \termvalue{u\approx v}^{\bf A}$.

\

(ii) $\termvalue{ x\in u}^{\bf A}=\bigvee\limits_{z\in dom(u)} (u(z) \wedge \termvalue{z \approx x}^{\bf A})\geq u(x) \wedge \termvalue{x \approx  x}^{\bf A}=  u(x)$. 

\end{proof}

\

As usual, we will adopt the following notation, for every formula $\varphi(x)$ and  every $u\in \mathbf{V}^{\langle A, N\rangle}$:  $\exists x\in u \varphi(x)= \exists x(x\in u \wedge \varphi(x))$ and $\forall x\in u \varphi(x)= \forall x(x\in u \to \varphi(x))$. 

Now, we recall that for a given Heyting algebra $A$ we have the following properties hold: (P1) $x\wedge y\leq z$ implies $x\leq y \Rightarrow z$  and (P2) $y\leq z$ implies $z \Rightarrow x \leq y \Rightarrow x$ for any $x,y,z\in A$.

Thus, we have the following

\begin{lemma}\label{BQ}
Let $A$ be a complete Heyting algebra, for every formula $\varphi(x)$ and every $u\in \mathbf{V}^{\bf A}$ we have: 
 $$\termvalue{ \exists x\in u \varphi(x)}^{\bf A}= \bigvee\limits_{x\in dom(u)} (u(x) \wedge \termvalue{ \varphi(x)}^{\bf A}),$$ 
 $$\termvalue{ \forall x\in u \varphi(x)}^{\bf A}= \bigwedge\limits_{x\in dom(u)} (u(x) \to \termvalue{\varphi(x)}^{\bf A}).$$
\end{lemma}
\begin{proof} The  proof of the case for the existential quantifier is identical to Bell's book, see \cite[Corollary 1.18, p. 27]{Bell}. On the other hand,

$\termvalue{ \forall x\in u \varphi(x)}^{\bf A}= \termvalue{\forall x(x\in u \to \varphi(x))}^{\bf A}= \bigwedge\limits_{v\in  \mathbf{V}^{\bf A}} \termvalue{v\in u}^{\bf A} \to \termvalue{\varphi(v)}^{\bf A}$

Then, we have

$ \bigwedge\limits_{x\in dom(u) } [u(x)\to \termvalue{\varphi(x)}^{\bf A}] \wedge \termvalue{v\in u}^{\bf A}=\bigwedge\limits_{x\in dom(u) } [u(x)\to \termvalue{\varphi(x)}^{\bf A}] \wedge \bigvee\limits_{x\in dom(u) } (u(x)\wedge  \termvalue{v\approx x}^{\bf A})= \bigvee\limits_{x\in dom(u) }(\bigwedge\limits_{x\in dom(u) } [u(x)\to \termvalue{\varphi(x)}^{\bf A}] \wedge u(x)\wedge  \termvalue{v\approx x}^{\bf A})\leq \bigvee\limits_{x\in dom(u) } \termvalue{\varphi(x)}^{\bf A} \wedge  \termvalue{v\approx x}^{\bf A}\leq \termvalue{\varphi(v)}^{\bf A}$

\

Form the latter and (P1), we can conclude that $ \bigwedge\limits_{x\in dom(u) } [u(x)\to \termvalue{\varphi(x)}^{\bf A}] \leq \termvalue{v\in u}^{\bf A}\to \termvalue{\varphi(v)}^{\bf A}$.

\

Now using Lemma 2.3 (ii) and (P2) we obtain 

$\bigwedge\limits_{v\in  \mathbf{V}^{\bf A}} \termvalue{v\in u}^{\bf A}\to \termvalue{\varphi(v)}^{\bf A} \leq \bigwedge\limits_{v\in  dom(u)} \termvalue{v\in u}^{\bf A}\to \termvalue{\varphi(v)}^{\bf A}\leq \bigwedge\limits_{v\in  dom(u)} \termvalue{u(v)}^{\bf A}\to \termvalue{\varphi(v)}^{\bf A}$. \end{proof}

\subsection{ Validating axioms}\label{VA}

Finally, it is possible to validate several set-theoretical axioms of  {\bf ZF}.  Indeed, let us consider a fixed model $\mathbf{V}^{ A}$.  The validity of (Separation) can be given in the same way that was done for Heyting-valued models case. Moreover, the validity of  (Pairing), (Union)  and (Collection) have the same proof that the intuitionistic case,  see Section 5.1 of \cite{FS0}.

\section{Nelson's Set Theories}

In the sequel, we will present the set theories $ZF_{\bf N4}$ and $ZF_{\bf N3}$, these theories are based on  the logics {\bf N4} and {\bf N3}. Firstly, we will focus on the development of $ZF_{\bf N4}$ ; and as a by-product, one is able to have the development of the theory $ZF_{\bf N3}$.

Let us construct now the class ${\mathbf{V}}^{\langle {\bf A}, N\rangle}$ of {\bf N4}-structure-valued model over $\langle {\bf A},N\rangle$ following the methodology developed for Heyting-valued models, see Section \ref{section4}. By ${\cal L}_\in$, we denote the first-order language of set theory which consists of the propositional connectives $\{\to, \wedge, \vee, \neg,\bot \}$ of the {\bf N4} and two binary predicates $\in$ and $\approx$. We can expand this language by adding all the elements of ${\mathbf{V}}^{\langle {\bf A}, N\rangle}$; the expanded language we will denote ${\cal L}_{\langle {\bf A}, N\rangle}$. Now, we will define a valuation by induction on the complexity of a closed formula in ${\cal L}_{\langle {\bf A}, N\rangle}$. Then, for a given complete and saturated {\bf N4}-structure $\langle {\bf A}, N\rangle$, the relation $\ter{\cdot}^{\langle {\bf A}, N\rangle}\subseteq {\cal L}_{\langle {\bf A}, N\rangle}\times A $, where ${\bf A}$ is the algebraic support of ${\langle {\bf A}, N \rangle}$, is defined as follows:{\small\begin{align}\termvalue{\bot}^{\langle {\bf A}, N\rangle} & = 0;\tag{N4-0}\\[3mm]\termvalue{u\in v}^{\langle {\bf A}, N\rangle} & =\bigvee\limits_{x\in dom(v)} (v(x) \wedge \termvalue{x \approx u}^{\langle {\bf A}, N\rangle});\tag{N4-1}\\[3mm]\ter{u\approx v}^{\langle {\bf A}, N\rangle} &=\bigwedge\limits_{x\in dom(u)} (u(x)) \to \ter{x\in v}^{\langle {\bf A}, N\rangle}) \wedge \bigwedge\limits_{x\in dom(v)} (v(x) \to \ter{x\in u}^{\langle {\bf A}, N\rangle});\tag{N4-2}\\[3mm]\ter{\varphi \# \psi}^{\langle {\bf A}, N\rangle} &= \ter{\varphi}^{\langle {\bf A}, N\rangle} \# \ter{\psi}^{\langle {\bf A}, N\rangle}\,\, \text{for every}\,\, \#\in \{\wedge,\vee, \to\};\tag{N4-3}\\[3mm]\ter{\neg \varphi}^{\langle {\bf A}, N\rangle}&\in N_{\ter{\varphi}^{\langle {\bf A}, N\rangle}}\,\,\text{and}\,\,\ter{\neg\neg\varphi}^{\langle {\bf A}, N\rangle}=\ter{\varphi}^{\langle {\bf A}, N\rangle};\tag{N4-4}\\[3mm]\ter{u\approx v}^{\langle {\bf A}, N\rangle} &\leq \ter{\neg\phi(u)}^{\langle {\bf A}, N\rangle}\to\ter{ \neg\phi(v)}^{\langle {\bf A}, N\rangle}\,\, \text{for any atomic formula}\,\, \phi(x).\tag{N4-5}\\[3mm]\ter{\neg (\varphi\vee \psi)}^{\langle {\bf A}, N\rangle}&= \ter{\neg \varphi}^{\langle {\bf A}, N\rangle} \wedge \ter{\neg \psi}^{\langle {\bf A}, N\rangle}\,\, \text{and}\,\,\ter{\neg(\varphi\wedge \psi)}^{\langle {\bf A}, N\rangle}=\ter{\neg \varphi}^{\langle {\bf A}, N\rangle}\vee \ter{\neg \psi}^{\langle {\bf A}, N\rangle};\tag{N4-6}\\[3mm]\ter{\neg(\varphi\to \psi)}^{\langle {\bf A}, N\rangle}&=\ter{\varphi}^{\langle {\bf A}, N\rangle}\wedge \ter{\neg \psi}^{\langle {\bf A}, N\rangle};\tag{N4-7}\\[3mm]\ter{\exists x\varphi}^{\langle {\bf A}, N\rangle} &= \bigvee\limits_{{u\in {\mathbf{V}}^{\langle A, N\rangle}}} \ter{\varphi (u)}^{\langle {\bf A}, N\rangle}\,\, \text{and}\,\, \ter{\forall x\varphi}^{\langle {\bf A}, N\rangle} = \bigwedge\limits_{{u\in {\mathbf{V}}^{\langle A, N\rangle}}} \ter{\varphi (u)}^{\langle {\bf A}, N\rangle};\tag{N4-8}.\end{align}}
\
\noindent Let us now consider a fixed but arbitrary mapping $\termvalue{\cdot}^{\bf A}$; i.e., the relation is taken as a function. Moreover, $\ter{\varphi}^{\langle {\bf A}, N\rangle}$ is called the possible {\bf truth-values} of the sentence $\varphi$ in the language ${\cal L}_{\langle {\bf A}, N\rangle}$ in the {\bf N4}-structure-valued model over $\langle {\bf A}, N\rangle$. It is important to note that $N_{\ter{\varphi}^{\langle {\bf A}, N\rangle}}=A$ and the condition (N4-5) defines the possible truth-values of the formulas with negation. Moreover, we need the saturated structures in order to have ``enough'' elements to guarantee that condition (N4-5) can be satisfied.

Now, we say that a sentence $\varphi$ in the language ${\cal L}_{\langle {\bf A}, N\rangle}$ is said to be valid in ${\mathbf{V}}^{\langle {\bf A}, N\rangle}$, which is denoted by ${\mathbf{V}}^{\langle {\bf A}, N\rangle} \vDash \varphi$, if $\ter{\varphi}^{\langle {\bf A}, N\rangle}=1$. 
For every completed N4-structure $\langle {\bf A}, N\rangle$, the element $\bigwedge\limits_{x\in A} x$ is the first element of $A$. So, ${\bf A}$ is a complete Heyting algebra, we denote this element by ''$0$''. Besides, for every closed formula $\phi$ of ${\cal L}_{\langle {\bf A}, N\rangle}$ we have $\ter{\phi}^{\langle {\bf A}, N\rangle} \in A$. Another issue worth mentioning is to observe that the condition $\ter{u\approx v}^{\langle {\bf A}, N\rangle} \leq \ter{\phi(u)}^{\langle {\bf A}, N\rangle}\to\ter{ \phi(v)}^{\langle {\bf A}, N\rangle}$ is held for any negation-free formula $\phi(x)$ in the setting of Heyting-valued models, where $\ter{\cdot}^{\langle {\bf A}, N\rangle}$ is a mapping. Similar to Section \ref{section4}, for every election the truth-value of each sentence with negation, we will take a fixed mapping but arbitrary.  It is clear that  $\ter{\cdot}^{\langle {\bf A}, N\rangle}$ represents all possible mappings that can be defined as truth-value for any sentence.

\begin{lemma}\label{sound}
For a given   {\bf N4}-structure $\langle {\bf A}, N\rangle$. Then, {\bf (PN1)}, {\bf (PN2)}, {\bf (PN3)} and {\bf (PN4)} are valid in ${\mathbf{V}}^{\langle {\bf A}. N\rangle}$; moreover, the logic {\bf N4} is sound in ${\mathbf{V}}^{\langle {\bf A}. N\rangle}$ 
\end{lemma}
\begin{proof}
It follows immediately from the condition (N4-4), (N4-6)  and (N4-7). 
\end{proof}

\

Now we are in conditions to show that   $\ter{\cdot}^{\langle {\bf A}, N\rangle}$ is well-defined for every saturated  {\bf N4}-structure $\langle {\bf A}, N\rangle$.  Firstly, we need to recall some properties of the theory of Heyting algebras.

\begin{rem}\label{imp}
For a given Heyting algebra $\bf A$ and a subset $D\subseteq A$, $D$ is to be said a deductive system (d.s.) if (i) $1\in D$ and (ii) $x,x\to y\in D$ imply $y\in D$. It is well-known that every deductive system is a filter and vice versa. We are here considering the standard notion of filter (\cite{Birk}) and every filter determines a unique congruence. So, for every filter (or d.s.)  $D$ we have the congruence $R(D)=\{(x,y)\in A^2: x\to y, y\to x \in D\}$ associated to it, see for instance \cite{AM}.
\end{rem}

The last Remark is important to show that our definition of $\ter{ \cdot }^{\langle {\bf A}, N\rangle}$ is consistent with the axioms of (paraconsistent) Nelson's logic. Indeed:

\begin{theo}\label{LL} For any complete and saturated {\bf N4}-structure the following property holds (Leiniz's law): $\ter{u\approx v}^{\langle {\bf A}, N\rangle} \leq \ter{\phi(u)\to \phi(v)}^{\langle {\bf A}, N\rangle}$ for any formula $\phi(x)$.
\end{theo}
\begin{proof}
If $\phi$ a negation-free formula, then Lemma is proved as the Heyting-valued models case. Clearly, if  $\phi(x)$ is the form $\neg\neg\psi(x)$ where $\psi(x)$ is not the form $\neg\beta(x)$ where $\beta(x)$ is a formula, then  Lemma is held by induction hypothesis, the axiom {\bf (PN1)} and from Lemma \ref{sound}.

Now, let us suppose that $\phi(x)$ is the form $\neg( \alpha (x) \vee \beta(x))$. Also, let $ u , v \in \langle {{\bf A}, N\rangle}$, then we will consider the following elements $\ter{u\approx v}^{\langle {\bf A}, N\rangle} =c$,  $\ter{\neg\alpha (u)}^{\langle {\bf A}, N\rangle}=a$ and  $\ter{\neg\beta(u)}^{\langle {\bf A}, N\rangle} =b$ and also, $\ter{\neg\alpha (v)}^{\langle {\bf A}, N\rangle} =a'$ and  $\ter{\neg\beta(v)}^{\langle {\bf A}, N\rangle}=b'$, with $c,a,b,a',b'\in A$. So, we have by induction hypothesis that $c\leq a \leftrightarrow a'$ and $c\leq b \leftrightarrow b'$. Taking the filter $F(c)=\{z\in A: c\leq z\}$ and having in mind Remark \ref{imp}, we have that $c\leq (a \wedge b) \leftrightarrow (a' \wedge b')$. Thus,  $\ter{u\approx v}\leq \ter{(\neg\alpha (u) \wedge \neg\beta (u))}^{\langle {\bf A}, N\rangle}  \leftrightarrow \ter{(\neg\alpha (v) \wedge \neg\beta (v))}^{\langle {\bf A}, N\rangle} $. From the latter and definition of $\ter{ \cdot }^{\langle {\bf A}, N\rangle}$ , we obtain that   $\ter{u\approx v}\leq \ter{\neg( \alpha (u) \vee \beta (u))}^{\langle {\bf A}, N\rangle}  \leftrightarrow \ter{\neg (\alpha (v) \vee\beta (v))}^{\langle {\bf A}, N\rangle} $ as desired. In a similar way, we obtain the cases derived from  the axioms of  {\bf (PN3)} and  {\bf (PN4)}. The rest of case for different forms of $\phi(x)$ follows immediately by  induction hypothesis.
\end{proof}

\

As corollary of the last Lemma, we have that the function $\ter{\cdot}$ is well-defined.  From the Lemmas \ref{BQ} and \ref{lema1} and using Theorem \ref{LL}, we have proven the following central results: 
  
\begin{lemma}
For a given   {\bf N4}-structure $\langle {\bf A}, N\rangle$. Then,  $\ter{u \approx u}^{\langle {\bf A}, N\rangle}=1$,  $u(x)\leq \ter{ x\in u}^{\langle {\bf A}, N\rangle}$ for every $x\in dom(u)$, and $\ter{u\approx v}^{\langle {\bf A}, N\rangle}=\ter{v\approx u}^{\langle {\bf A}, N\rangle}$, for every $u,v\in {\mathbf{V}}^{\langle {\bf A}, N\rangle}$.
\end{lemma}
 
\begin{lemma}
Let $\langle {\bf A}, N\rangle$ be a  complete and saturated {\bf N4}-structure. The, for every formula $\varphi(x)$ and every $u\in \mathbf{V}^{\langle A, N\rangle}$, we have: 

 $$\ter{\exists x\in u \varphi(x)}^{\langle {\bf A}, N\rangle}= \bigvee\limits_{x\in dom(u)} (u(x) \wedge \ter{\varphi(x)}^{\langle {\bf A}, N\rangle}),$$ 
 $$\ter{ \forall x\in u \varphi(x)}^{\langle {\bf A}, N\rangle}= \bigwedge\limits_{x\in dom(u)} (u(x) \to \ter{\varphi(x)}^{\langle {\bf A}, N\rangle}).$$
\end{lemma}

The following definitions and lemmas are presented in order to exhibit the proof of the validity of axiom (Infinity) using the standard empty set given by the axiom (Empty Set). Indeed:

\begin{defi}
Let $\langle {\bf A}, N\rangle$ be a  complete and saturated {\bf N4}-structure. Given collection of sets $\{u_i: i\in I\} \subseteq \mathbf{V}^{ \langle {\bf A}, N\rangle}$ and $\{a_i: i\in I \}\subseteq A$, then mixture $\Sigma_{i\in I} a_i\cdot u_i$ is the function $u$ with $dom(u)=\bigcup\limits_{i\in I} dom(u_i)$ and, for $x\in dom(u)$, $u(x)= \bigvee\limits_{i\in I} a_i \wedge \termvalue{x\in u_i}^{\langle {\bf A}, N\rangle}$.
\end{defi}

The following result is known as {\em Mixing Lemma} and its proof is exactly the same for the intuitionistic case because it is an assertion about positive formulas. 

\begin{lem}{\bf {\rm [}The Mixing Lemma{\rm ]}}\label{maximlemma}
Let $u$ be the mixture $\Sigma_{i\in I} a_i\cdot u_i$. If $a_i\wedge a_j\leq \termvalue{ u_i \approx u_j}^{\langle {\bf A}, N\rangle}$ for all $i,j\in I$, then $a_i\leq \termvalue{u_i \approx u}^{\langle {\bf A}, N\rangle}$.  
\end{lem}

 A set $B$ refines a set $A$ if for all $b\in B$ there is some $a\in A$ such that $b\leq a$. A Heyting algebra $H$ is refinable if every subset $A\subseteq H$ there exists some anti-chain  $B$ in $H$ that refines $A$ and verifies $\bigvee A = \bigvee B$.

\begin{theo}\label{maximumprinciple}
Let $\langle {\bf A}, N\rangle$ be a  complete and saturated {\bf N4}-structure such that $A$ is  refinable. If $\mathbf{V}^{\langle {\bf A}, N\rangle}\vDash \exists x \psi(x)$, then there is $u\in \mathbf{V}^{ \langle {\bf A}, N\rangle}$ such that $\mathbf{V}^{\langle {\bf A}, N\rangle}\vDash \psi(u)$. 
\end{theo}

\begin{proof} The proof runs exactly the same to the one given in \cite[Theorem 5.5]{FS0}.
\end{proof}

\

Now, given a complete Heyting subalgebra $ {\bf A'}$ of ${\bf A}$, we have the associated models $\mathbf{V}^{\bf A'}$ and $\mathbf{V}^{\bf A}$. Then, it is easy to see that $\mathbf{V}^{ {\bf A'}}\subseteq \mathbf{V}^{\bf A}$. On the other hand, we say that a formula $\psi$ is {\em restricted} if all quantifiers are of the form $\exists y\in x$ or $\forall y\in x$, then we have  

\begin{lem}
For any complete Heyting subalgebra $ \bf A'$ of $ \bf A$ and any restricted negation-free formula $\psi(x_1,\cdots,x_n)$ with variables in  $\mathbf{V}^{ \bf A'}$ the equality $\termvalue{\psi(x_1,\cdots,x_n)}^{\bf A'}=\termvalue{\psi(x_1,\cdots,x_n)}^{\bf A}$.
\end{lem}

\

Next, we are going to consider the Boolean algebra ${\bf 2}=(\{0,1\},\wedge,\vee,\neg,0,1)$ and the natural mapping $\hat{\cdot }:  \mathbf{V}^{\bf A} \to \mathbf{V}^{\bf A}$ defined by $\hat{u}=\{\langle \hat{v}, 1\rangle: v\in u\}$. This is well defined by recursion on $v\in dom(u)$. Then, we have the following lemma:

\begin{lem}\label{tecnico}
\begin{itemize}
\item[\rm (i)] $\termvalue{u\in \hat{v}}^{\bf A}=\bigvee\limits_{x\in v} \termvalue{ u\approx\hat{x}}^{\bf A} $ for all $v\in \mathbf{V}$ and $u\in \mathbf{V}^{\bf A }$,

\item[\rm (ii)] $u\in v \leftrightarrow \mathbf{V}^{\bf A } \vDash \hat{u}\in \hat{v}$ and $u\approx v \leftrightarrow \mathbf{V}^{\bf A } \vDash \hat{u}\approx \hat{v}$,

\item[\rm (iii)] for all $x\in \mathbf{V}^{ {\langle {\bf A}, N \rangle}}$ there exists a unique $v\in \mathbf{V}$ such that $\mathbf{V}^{ {\bf 2}}\vDash  x\approx \hat{v}$,

\item[\rm (iv)] for any formula negation-free formula $\psi(x_1,\cdots,x_n)$  and any $x_1, \cdots,x_n\in \mathbf{V}$, we have  $\psi(x_1,\cdots,x_n) \leftrightarrow \mathbf{V}^{ {\langle {\bf A}, N \rangle}} \vDash   \psi(\hat{x_1},\cdots,\hat{x_n})$. Moreover for any restricted negation-free formula $\phi$, we have $\phi(x_1,\cdots,x_n) \leftrightarrow \mathbf{V}^{\bf A} \vDash   \phi(\hat{x_1},\cdots,\hat{x_n})$. 

\end{itemize}

\end{lem}

The proof of the last theorem is the same for the intuitionistic case, because we consider restricted negation-free formulas.

 \

In the following Theorem, we will consider a complete and saturated {\bf N4}-structure $\langle {\bf A}, N\rangle$, then $\bf A$ is indeed a complete Heyting algebra. Moreover, for every formula $\psi$, we have that $\ter{\psi}^{\langle {\bf A}, N\rangle}$ belongs to $A$. From the latter and  taking into account the proofs displayed on  Section \ref{VA}, we have proven this following Theorem: 

\begin{theo}\label{ZF1} Let $\langle {\bf A}, N\rangle$ be a complete and saturated  {\bf N4}-structure. Then, the  set-theoretic axioms  (Pairing), (Union),  (Separation),  (Collection), (Empty set) and  (Infinity) are valid in ${\mathbf{V}}^{\langle {\bf A}, N\rangle}$.
\end{theo}
\begin{proof}
The validity of (Pairing), (Union),  (Separation),  (Collection) can be proven in exactly the same way that was made in Section \ref{VA}. Clearly,  (Empty set) is valid using the standard proof for the Heyting-valued models. Thus, the (Infinity) axiom can be proved by using the intuitionistic argument because of the (Empty set) is expressing with a intuitionistic negation (also it works for the classical one) and since the axioms (Empty set) and  (Infinity) are restricted formulas. Using now  Lemma \ref{tecnico}, we have the axiom (Infinity) is valid as the intuitionistic case. 
\end{proof}

\

Now, we are in conditions of proving the axioms (Extensionality),  (Powerset) and (induction) which their proofs run exactly as the intuitionistic case. Indeed:

\begin{theo}\label{ZF2} Let $\langle {\bf A}, N\rangle$ be a complete and saturated {\bf N4}-structure. Then, the  set-theoretic axioms  (Extensionality),  (Powerset) and  and (Induction)  and   are valid in ${\mathbf{V}}^{\langle {\bf A}, N\rangle}$.
\end{theo}
\begin{proof} The proof is exactly the same for Heyting-valued models, see for instance \cite{FS0}. 
\end{proof}

\

Then we are in conditions to present the main result of this section: 

\begin{coro}\label{N4}  All  set-theoretic axioms  $ZF_{\bf N4}$  are valid in ${\mathbf{V}}^{\langle {\bf A}, N\rangle} $. 
\end{coro}

\

\subsection{ $ZF_{\bf N3}$ Set Theory}
It is clear that  the proofs on $ZF_{\bf N4}$ we can be translated  to the case of $ZF_{\bf N3}$. So,  for a given complete and saturated {\bf N3}-structure $\langle {\bf A}, N\rangle$,  the mapping $\ter{\cdot}^{\langle {\bf A}, N\rangle}:{\cal L}_{\langle {\bf A}, N\rangle}\to \langle {\bf A}, N\rangle $ is defined as {\bf N4} case adding the condition: for every sentence $\alpha$, let us permit to take $\ter{\alpha \wedge \neg \alpha}^{\langle {\bf A}, N\rangle}=0$.  For the rest of sub-section, we will give a proof for arbitrary, but fixed, mapping $\ter{\cdot}^{\langle {\bf A}, N\rangle}$.
 
\begin{lemma}\label{414}  For any complete and saturated {\bf N4}-structure the following property holds (Leiniz's law): $\ter{u\approx v}^{\langle {\bf A}, N\rangle} \leq \ter{\phi(u)\to \phi(v)}^{\langle {\bf A}, N\rangle}$ for any formula $\phi(x)$ and $\ter{\cdot}^{\langle {\bf A}, N\rangle}$ is well-defined.
\end{lemma}
\begin{proof}
Taking into account the proof of Theorem \ref{LL} and {\bf (Ax9)}, we only have to see if the formula $\phi(x)$ is the form $\alpha(x) \wedge \neg \alpha(x)$, then it verifies Leiniz's law. But it is  immediate from the fact that  $\ter{\phi(u)}^{\langle {\bf A}, N\rangle}=0$ for any $u\in{\mathbf{V}}^{\langle {\bf A}, N\rangle}$. Furthermore, the axiom {\bf (Ax9)} $(\alpha \wedge \neg\alpha)\to \beta$ is sound on ${\mathbf{V}}^{\langle {\bf A}, N\rangle}$ by using algebraic properties of Heyting algebras and the definition of {\bf N3}-structure.
\end{proof}

\begin{coro}\label{N3}  All  set-theoretic axioms  $ZF_{\bf N3}$  are valid in ${\mathbf{V}}^{\langle {\bf A}, N\rangle} $.
\end{coro}
\begin{proof}
It is immediate from the proofs for $ZF_{\bf N4}$ and Lemma \ref{414}.
\end{proof}

\subsection{Paraconsistent empty set}

If we consider the axiom: 

\begin{axioms}
\AxiomItem{(Paraconsistent Empty set)} {\exists x\forall z[z\in x \leftrightarrow  \neg(z\approx z)]},
\end{axioms}

Now, we show that this axiom  is valid. Indeed, first let us note that $\ter{u\approx u}^{\langle {\bf A}, N\rangle}=1$ for all $u\in \mathbf{V}^{\langle {\bf A }, N\rangle}$ and then, $\ter{\neg(u\approx u)}^{\langle {\bf A}, N\rangle}\in N_1$, taking $\ter{\cdot}^{\langle {\bf A}, N\rangle}$ as mapping. Therefore, let us consider a function $w\in \mathbf{V}^{\langle {\bf A }, N\rangle}$ such that  $u\in dom(w)$ and $\emptyset\not=ran(w)\subseteq \{\ter{\neg(u\approx u)}^{\langle {\bf A}, N\rangle}\}$, where $\emptyset$ is the empty set of the meta-system. Then, it is clear that $\ter{u\in w}^{\langle {\bf A}, N\rangle} = \bigvee\limits_{x\in dom(w)} (w(x)\wedge \ter{u\approx x}^{\langle {\bf A}, N\rangle})=\ter{\neg(u\approx u)}^{\langle {\bf A}, N\rangle}$  which completes the proof.

\section{$ZFC^2_\omega$ and $ZF\Cw$ Set Theories}\label{Sc1} 

Now, let us construct the class ${\mathbf{V}}^{\langle {\bf A}, N \rangle}$  of $C^2_\omega$-structure valued  model over $C^2_\omega$ following Section \ref{section4} but now using  the complete Boolean algebra ${\bf A}$. We will take the class of complete Boolean algebras because of it is term-equivalent to the class of complete Tarski algebras with infimum, see justification in Section \ref{Cw}. By ${\cal L}_\in$, we denote the first-order language of set theory which consists of the propositional connectives $\{\to, \wedge, \vee, \neg, \bot\}$ for the logic $C^2_\omega$  and two binary predicates $\in$ and $\approx$. Let us now extend  this language by adding all the elements of ${\mathbf{V}}^{\langle {\bf A}, N \rangle}$; the extended language we will denote ${\cal L}_{\bf 2}$.  We will now  define a valuation by induction on the complexity of a closed formula in ${\cal L}_{\bf 2}$. Then,   the relation $\ter{\cdot}^{\langle {\bf A}, N \rangle}\subseteq {\cal L}_{\bf 2}\times {\bf A}$, where ${\bf A}$ is the algebraic support of ${\langle {\bf A}, N \rangle}$, is defined as follows:

{\small
\begin{align}
\termvalue{\bot}^{\langle {\bf A}, N \rangle}& =0;\tag{V0}\\[3mm]
\termvalue{u\in v}^{\langle {\bf A}, N \rangle}& =\bigvee\limits_{x\in dom(v)} (v(x) \wedge \termvalue{x \approx u}^{\langle {\bf A}, N \rangle});\tag{V1}\\[3mm]
\termvalue{u\approx v}^{\langle {\bf A}, N \rangle}&=\bigwedge\limits_{x\in dom(u)} (u(x)) \to \termvalue{x\in v}^{\langle {\bf A}, N \rangle}) \wedge \bigwedge\limits_{x\in dom(v)} (v(x) \to \termvalue{x\in u}^{\langle {\bf A}, N \rangle});\tag{V2}\\[3mm]
\ter{\neg \varphi}^{\langle {\bf A}, N\rangle}&\in N_{\ter{\varphi}^{\langle {\bf A}, N\rangle}};\tag{V3}\\[3mm]
\termvalue{\varphi \# \psi}^{\langle {\bf A}, N \rangle} & = \termvalue{\varphi}^{\langle {\bf A}, N \rangle}\# \termvalue{\psi}^{\langle {\bf A}, N \rangle}, \text{for every}\,\, \#\in \{\wedge,\vee, \to\};\tag{V4}\\[3mm]
\ter{u\approx v}^{\langle {\bf A}, N \rangle}&\leq \ter{\neg\phi(u)}^{\langle {\bf A}, N \rangle}\to \ter{\neg\phi(v)}^{\langle {\bf A}, N \rangle}\,\,\text{for any  formula}\,\, \phi.\tag{V5}\\[3mm]
\ter{\exists x\varphi}^{\langle {\bf A}, N\rangle} &= \bigvee\limits_{{u\in {\mathbf{V}}^{\langle A, N\rangle}}} \ter{\varphi (u)}^{\langle {\bf A}, N\rangle}\,\, \text{and}\,\, \ter{\forall x\varphi}^{\langle {\bf A}, N\rangle} = \bigwedge\limits_{{u\in {\mathbf{V}}^{\langle A, N\rangle}}} \ter{\varphi (u)}^{\langle {\bf A}, N\rangle};\tag{V6}.
\end{align}
}
\

\noindent As done in the previous sections, let us now consider a fixed but arbitrary mapping $\termvalue{\cdot}^{\bf A}$; i.e., the relation is taken as a function. In this setting, $\termvalue{\varphi}^{\langle {\bf A}, N \rangle}$ is called the possible {\bf truth-values} of the sentence $\varphi$ in the language ${\cal L}_{\bf 2}$ in the $C^2_\omega$-structure valued  model  over ${\langle {\bf A}, N \rangle}$. Now, we say that a sentence $\varphi$ in the language ${\cal L}_{\bf 2}$  is said to be valid in ${\mathbf{V}}^{\langle {\bf A}, N \rangle}$, which  is denoted by ${\mathbf{V}}^{\langle {\bf A}, N \rangle}\vDash \varphi$, if $\termvalue{\varphi}^{\langle {\bf A}, N \rangle}=1$.

  It is worth mentioning that $\termvalue{\cdot}^{\langle {\bf A}, N \rangle}$  behaves as homomorphism  for negation-free formulas. Moreover, for these formulas the valuation is a mapping. Now, for every formula with the form $\neg \alpha$ we have that $\termvalue{\neg \alpha}^{\langle {\bf A}, N \rangle}\in N_{\termvalue{\alpha}^{\langle {\bf A}, N \rangle}}$. The last condition is under the hypothesis that $\termvalue{\cdot}^{\langle {\bf A}, N \rangle}$  is a function, thus  it is well-defined in spite of the fact that we do not know that which is the real value taken by  $\termvalue{\neg \alpha}^{\langle {\bf A}, N \rangle}$. Recall that if we are working with a saturated $C^2_\omega$-structure $\langle {\bf A}, N\rangle$, so we have that $1\in N_a$ for every $a\in A$ and $N_1=A$; moreover, if $\bf A$ is a Boolean algebra we have that $1,\neg a\in N_a$ for every $a\in A$, where $\neg$ is the Boolean negation.  In the next,  let us consider $\termvalue{\cdot}^{\langle {\bf A}, N \rangle}$ as a function with a fixed value truth-value for formulas of the type $\neg \alpha$, it will be fixed but arbitrary. It is clear that the sets  $N_{\termvalue{\alpha}^{\langle {\bf A}, N \rangle}}$ are determined by induction on the structure of formula $\alpha$ as the propositional level, see Section \ref{Cw}. At the same time, we will consider the class ${\mathbf{V}}^{\langle {\bf A}, N\rangle}$ together with this valuation $\termvalue{\cdot}^{\langle {\bf A}, N \rangle}$. For every mention  of ${\mathbf{V}}^{\langle {\bf A}, N\rangle}$, the valuation is taken implicitly. Thus, the following Lemma has the same proof that  Boolean case:

\begin{lemma}\label{lematec}
The following properties true in ${\mathbf{V}}^{\langle {\bf A}, N \rangle}$ for every  $C^2_\omega$-structure  ${\langle {\bf A}, N \rangle}$:
\begin{itemize}

\item[\rm (i)]   $\termvalue{u \approx u}^{\langle {\bf A}, N \rangle}=1$,  
\item[\rm (ii)] $u(x)\leq \termvalue{ x\in u}^{\langle {\bf A}, N \rangle}$ for every $x\in dom(u)$, 
\item[\rm (iii)]  $\termvalue{u\approx v}^{\langle {\bf A}, N \rangle}=\termvalue{v\approx u}^{\langle {\bf A}, N \rangle}$, for every $u,v\in {\mathbf{V}}^{\langle {\bf A}, N \rangle}$
\end{itemize}

\end{lemma}
\begin{proof}
(i), (ii) and (iii): It is immediately from \cite[Theorem 1.7]{Bell} an the fact that they are properties for positive formulas.
\end{proof}

\begin{lemma}
For every   $C^2_\omega$-structure    valued  mode ${\mathbf{V}}^{\langle {\bf A}, N \rangle}$, then the logics $C_\omega$ and $C^2_\omega$ are sound w.r.t.  $\ter{\cdot}^{\langle {\bf A}, N \rangle}$.
\end{lemma}
\begin{proof}
It is enough to see that the axioms {\bf (C1)}, {\bf (C2)} and {\bf  (C3)} are negation free formulas and they are sound on every Boolean algebra.
\end{proof}

\begin{lemma}{\rm (\cite{FS0})}
Let  ${\mathbf{V}}^{\langle {\bf A}, N \rangle}$ be a  $C^2_\omega$-structure  valued  model, we have that there is a interpretation of   $\ter{\cdot}^{\langle {\bf A}, N \rangle}$ with  paraconsistent behavior.
\end{lemma}
\begin{proof}
First, let us consider the saturated structure ${\langle {\bf A}, N \rangle}$; i.e., $N_1=A$ and  we always have $1\in N_x$ for every $x\in A$. Take now  $u\in {\mathbf{V}}^{\langle {\bf A}, N \rangle}$ and a closed formula $\neg\phi(u)$ where $\phi(u)$ is not the form $\neg \psi(u)$ for some formula $\psi(x)$. Now, we define $\ter{\neg\phi(u)}^{\langle {\bf A}, N \rangle}=1$ and $\ter{\neg\neg\phi(u)}^{\langle {\bf A}, N \rangle}=\ter{\phi(u)}^{\langle {\bf A}, N \rangle}$; it clear that  $\ter{\neg\neg\phi(u)}^{\langle {\bf A}, N \rangle}\in N_1$. Thus, by induction of the complexity of $\phi(u)$ we have that  $\ter{\cdot}^{\langle {\bf A}, N \rangle}$ verifies $\ter{u\approx v}^{\langle {\bf A}, N \rangle}\leq \ter{\phi(u)}^{\langle {\bf A}, N \rangle}\to \ter{\phi(v)}^{\langle {\bf A}, N \rangle}$ for any  formula $\phi$ and it is a  $C^2_\omega$-valuation (also $C_\omega$-valuation), see Section \ref{Cw}. 
\end{proof}

\

\begin{lemma}\label{crucial}
Let $\phi(x)$ be a formula such that ${\mathbf{V}}^{\langle {\bf A}, N \rangle}\vDash \exists x\phi(x)$. Then,
\begin{itemize}
\item[\rm (i)] For any $v\in {\mathbf{V}}^{\langle {\bf A}, N \rangle}$ there is $u\in {\mathbf{V}}^{\langle {\bf A}, N \rangle}$ such that $\ter{\phi(u)}^{\langle {\bf A}, N \rangle}=1$ and $\ter{\phi(v)}^{\langle {\bf A}, N \rangle}=\ter{u\approx v}^{\langle {\bf A}, N \rangle}$.
\item[\rm (ii)] If $\psi(x)$ is a formula such that for any $u\in {\mathbf{V}}^{\langle {\bf A}, N \rangle}$, the condition ${\mathbf{V}}^{\langle {\bf A}, N \rangle}\vDash \phi(u)$ implies ${\mathbf{V}}^{\langle {\bf A}, N \rangle}\vDash \psi(u)$, then ${\mathbf{V}}^{\langle {\bf A}, N \rangle}\vDash \exists x[\phi(x)\to \psi(x)]$.
\end{itemize}
\end{lemma}

\begin{proof}
(i): Applying Theorem \ref{maximumprinciple}  to the Boolean algebra ${\bf A}$, we obtain that there is $w\in {\mathbf{V}}^{\langle {\bf A}, N \rangle}$ such that $\ter{\phi(w)}^{\langle {\bf A}, N \rangle}=1$. Now, taking $b=\ter{\phi(v)}^{\langle {\bf A}, N \rangle}$ and the mixing function $u=(b\cdot v)+ (b^* \cdot w)$ where $b^*$ is the Boolean complement of $b$. So, from Lemma \ref{lematec} (iv) we have: $\ter{\phi(u)}^{\langle {\bf A}, N \rangle}\geq \ter{u\approx v\wedge \phi(v)}^{\langle {\bf A}, N \rangle}\vee \ter{u\approx w\wedge \phi(w)}^{\langle {\bf A}, N \rangle}$. On the other hand, applying {\em Mixing Lemma} (Lemma \ref{maximlemma}) for the set $\{b,b^*\}$, we have $1=b\vee b^*\leq \ter{u\approx v\wedge \phi(v)}^{\langle {\bf A}, N \rangle}\vee \ter{u\approx w\wedge \phi(w)}^{\langle {\bf A}, N \rangle}$. Form the latter and (iv) again, $\ter{u\approx v}^{\langle {\bf A}, N \rangle}=\ter{u\approx v}^{\langle {\bf A}, N \rangle}\wedge \ter{\phi(u)}^{\langle {\bf A}, N \rangle}\leq \ter{\phi(v)}^{\langle {\bf A}, N \rangle}$. Since $b=\ter{\phi(v)}^{\langle {\bf A}, N \rangle}\leq \ter{u\approx v}^{\langle {\bf A}, N \rangle}$ which completes the proof.

(ii): Let $u\in {\mathbf{V}}^{\langle {\bf A}, N \rangle}$. From (i), there exists $u\in {\mathbf{V}}^{\langle {\bf A}, N \rangle}$ such that $\ter{\phi(u)}^{\langle {\bf A}, N \rangle}=1$ and $\ter{\phi(v)}^{\langle {\bf A}, N \rangle}=\ter{u\approx v}^{\langle {\bf A}, N \rangle}$. Thus, $\ter{\phi(v)}^{\langle {\bf A}, N \rangle}\leq \ter{\phi(v)}^{\langle {\bf A}, N \rangle}$ as desired. 
\end{proof}

\

Next, we will define the notion of {\em core} of an element of ${\mathbf{V}}^{\langle {\bf A}, N \rangle}$. 

\begin{definition} Let $u\in {\mathbf{V}}^{\langle {\bf A}, N \rangle}$, a set $v\subseteq {\mathbf{V}}^{\langle {\bf A}, N \rangle}$ is call core for $u$ if the following holds: (i) $\ter{x\in u}^{\langle {\bf A}, N \rangle}=1$ for all $x\in v$ and (ii) for each $y\in {\mathbf{V}}^{\langle {\bf A}, N \rangle}$  such that $\ter{y\in u}^{\langle {\bf A}, N \rangle}=1$ there is a ``unique'' $x\in v$ such that $\ter{x\approx y}^{\langle {\bf A}, N \rangle}=1$. 
\end{definition}

\begin{lemma}\label{core}
Any  $u\in {\mathbf{V}}^{\langle {\bf A}, N \rangle}$ has a core.
\end{lemma}
\begin{proof} Let us consider $a_x=\{\langle z, u(z)\wedge \ter{z\approx x}^{\langle {\bf A}, N \rangle}\rangle: z\in dom(u)\}$ for every $x\in {\mathbf{V}}^{\langle {\bf A}, N \rangle}$.  By {\em Collection} on the set $\{f\in {\bf 2}^{dom(u)}:\,\, \text{\em there is}\,\, z\in {\mathbf{V}}^{\langle {\bf A}, N \rangle}\,\,\text{\em such that}\,\, f=a_y\}$, there is  $w\subseteq  {\mathbf{V}}^{\langle {\bf A}, N \rangle}$ such that for all $y\in {\mathbf{V}}^{\langle {\bf A}, N \rangle}$ there exists $x\in w$ such that $a_x=a_y$. Let  now $C=\{x\in w: \ter{x\in u}^{\langle {\bf A}, N \rangle}=1\}$ be a set, for any $y$ if $x$ is such that $a_x=a_y$ then $u(z)\wedge \ter{z\approx x}^{\langle {\bf A}, N \rangle}= u(z)\wedge \ter{z\approx y}^{\langle {\bf A}, N \rangle}$ for all $z\in dom(u)$. Therefore, $\ter{y\in u}^{\langle {\bf A}, N \rangle}=1$ and then $x\in w$ such that $1=\bigvee\limits_{z\in dom(u)} [u(z)\wedge \ter{z\approx y}^{\langle {\bf A}, N \rangle}] = \bigvee\limits_{z\in dom(u)} [u(z)\wedge \ter{z\approx y}^{\langle {\bf A}, N \rangle}\wedge \ter{x\approx y}^{\langle {\bf A}, N \rangle}]\leq  \ter{x\approx y}^{\langle {\bf A}, N \rangle}$, by Lemma \ref{lematec} (iv). So, $x\in C$ as desired.

\end{proof}

\begin{coro}\label{coro}
Suppose that $u \in {\mathbf{V}}^{\langle {\bf A}, N \rangle}$ is such that ${\mathbf{V}}^{\langle {\bf A}, N \rangle}\vDash u\not=\emptyset$ and let $v$ be a core for $u$. Then for any $x\in {\mathbf{V}}^{\langle {\bf A}, N \rangle}$ there is $y\in v$ such that $\ter{x\approx y}^{\langle {\bf A}, N \rangle}=\ter{x\in u}^{\langle {\bf A}, N \rangle}$.
\end{coro}
\begin{proof}
It immediately follows from Lemma \ref{crucial} and taking $\phi(x)$:= ``$x\in u $''.
\end{proof}

\

In order to prove that the {\em Axiom of Choice} is true in ${\mathbf{V}}^{\langle {\bf A}, N \rangle}$, it is enough to verify the set-theoretically equivalent the  Zorn's Lemma principle. Recall now that a partially ordered set $X$  is said to be {\em inductive} if every chain in $X$ has upper bounds. Zorn's Lemma states that any non-empty inductive partially ordered set has a maximal element.

\begin{lemma}\label{ZornL1} 
Zorn's Lemma is true in ${\mathbf{V}}^{\langle {\bf A}, N \rangle}$.
\end{lemma}
\begin{proof}
Let $x,\leq_x\in{\mathbf{V}}^{\langle {\bf A}, N \rangle}$ where   $x$ is a set and $\leq_x$ an ordered relation on $x$. Let us now suppose the formula $\phi(x)$ is given by ''{\em $(x,\leq_x)$ is a non-empty inductive partially ordered set}'' and the formula $\psi(x)$ given by ''{\em $(x,\leq_x)$ has a maximal element}''.

 From Lemma \ref{crucial}(ii), we only prove that:  there is $u\in{\mathbf{V}}^{\langle {\bf A}, N \rangle}$ such that {\bf the condition} ${\mathbf{V}}^{\langle {\bf A}, N \rangle}\vDash \phi(u)$ {\bf  implies}  ${\mathbf{V}}^{\langle {\bf A}, N \rangle}\vDash \psi(u)$.
 Indeed, let us suppose $y$ is a core for $x$ and take the relation $\leq_y$ on $y$ defined by  ``$z\leq_y z'$ iff $\ter{z\leq_X z'}=1$'' for $z,z'\in y$. It is not hard to see that $\leq_y$ is partial ordering on $y$. We will prove that every chain in $y$ has upper bound in it. Indeed, let $C$ a chain in $y$ and take $C'=\{\langle c,1\rangle: c\in C\}$.  Then, it is clear that ${\mathbf{V}}^{\langle {\bf A}, N \rangle}\vDash C'\,\,\text{\em is a chain in}\,\, x$. From the latter and Theorem \ref{maximumprinciple} applied to $\bf 2$, we have that there is $u\in{\mathbf{V}}^{\langle {\bf A}, N \rangle}$ such that  ${\mathbf{V}}^{\langle {\bf A}, N \rangle}\vDash u\,\,\text{\em is a upper bound for}\,\, C'\,\,\text{\em in}\,\, x$. Taking $w\in y$ such that $\ter{w\approx u}^{\langle {\bf A}, N \rangle}=1$. Thus, $w$ is an upper bound for $C$ in $y$. For every $x\in C$, we have $\ter{x\in C'}^{\langle {\bf A}, N \rangle}=1$. Then, $\ter{x\leq_x u}^{\langle {\bf A}, N \rangle}=1$ and so $\ter{x\leq_x w}^{\langle {\bf A}, N \rangle}=1$. Therefore, $x\leq_y w$ and then $y$ is inductive. By Zorn's Lemma in ${\mathbf{V}}^{\langle {\bf A}, N \rangle}$, $y$ has maximal element $c$. We claim  $\ter{c\in x}^{\langle {\bf A}, N \rangle}=1$  and   ${\mathbf{V}}^{\langle {\bf A}, N \rangle}\vDash c\,\,\text{\em is a maximal element of}\,\, x$.
Indeed, let us take $t\in {\mathbf{V}}^{\langle {\bf A}, N \rangle}$ and apply Corollary \ref{coro}, there is $z\in y$ such that $\ter{t\in x}^{\langle {\bf A}, N \rangle}=\ter{t\approx z}^{\langle {\bf A}, N \rangle}$.
 So, $\ter{c\leq_x t \wedge t\in x}^{\langle {\bf A}, N \rangle}=\ter{c\leq_x t \wedge t\approx z}^{\langle {\bf A}, N \rangle}\leq \ter{c\leq_x z}^{\langle {\bf A}, N \rangle}$, by Leibniz's law. Let us now $u$ the mixture $a\cdot z + a^*\cdot c$ where $a=\ter{c\leq_x t}^{\langle {\bf A}, N \rangle}$ and  $a^*$ is the Boolean complement of $a$. Then, by Mixing Lemma  we have $ a\leq \ter{t\approx u}^{\langle {\bf A}, N \rangle}$ and $ a^*\leq \ter{c\approx u}^{\langle {\bf A}, N \rangle}$. So, $ \ter{u\in x}^{\langle {\bf A}, N \rangle}=1$. From the latter and  by the definition of $y$, there is $k\in y$ such that $ \ter{k\approx u}^{\langle {\bf A}, N \rangle}=1$. On the other hand,  it is not hard to see that $ \ter{c\leq_x u}^{\langle {\bf A}, N \rangle}=1$. So, by Leibniz law, we have  $ \ter{c\leq_x k}^{\langle {\bf A}, N \rangle}=1$ and then $c\leq_y k$. Hence, $c=k$ by the maximality of $c$. Therefore,  $\ter{c\leq_x z}^{\langle {\bf A}, N \rangle}\leq \ter{c\approx z}^{\langle {\bf A}, N \rangle}$  and so  $\ter{c\leq_x t \wedge t\in x}^{\langle {\bf A}, N \rangle}\leq  \ter{t\approx c}^{\langle {\bf A}, N \rangle}$. Thus, ${\mathbf{V}}^{\langle {\bf A}, N \rangle}\vDash \forall t\in x[ (c\leq_x t)  \to (t\approx c)  ]$ as desired.\end{proof}

\begin{theo}\label{C1}  All  set-theoretic axioms  $ZF_{C^2_\omega}$  plus axiom of choice (AC) are valid in ${\mathbf{V}}^{\langle {\bf A}, N \rangle}$.
\end{theo}
\begin{proof}
The proofs of the validity of  axioms  (Pairing), (Collection), (Separation), (Empty set), (Union), (Infinity) and (Induction) are the same the ones given in Section \ref{VA}, using Leibniz law and taking the structure ${\mathbf{V}}^{\langle {\bf A}, N \rangle}$. The validity of axioms    (Extensionality), (Powerset) and (Empty set) immediately follows from Theorem \ref{ZF2}. The validity of (AC) follows from Lemma \ref{ZornL1} and the fact that on $ZF_{C_1}$ the  Law of Excluded Middle is held.
\end{proof}

\

\begin{theo}  All  set-theoretic axioms    $ZF_{C_\omega}$  are valid in ${\mathbf{V}}^{\langle {\bf A}, N\rangle}$ for every completed and saturated $C^2_\omega$-structure ${\langle {\bf A}, N \rangle}$.
\end{theo}

In the book \cite[pag. 70]{Bell2}, it was presented a proof the Zorn's Lemma is valid on every Heyting-valued model.  The form of this Lemma, which is taken in this book, is in fact a stronger version but still strong enough to prove (AC). Adapting the proofs for the intuitionistic case and using Leibniz's law, we have proved the following Lemma. 

\begin{lemma} 
Zorn's Lemma is true in ${\mathbf{V}}^{\langle {\bf A}, N\rangle}$ and so (AC) is true  for every $C^2_\omega$-structure valued model ${\mathbf{V}}^{\langle {\bf A}, N\rangle}$.
\end{lemma}

\begin{coro}\label{coro52}
All  set-theoretic axioms  $ZF_{C_\omega}$ and $ZF_{C^2_\omega}$  are valid for every  Boolean-valued models.
\end{coro}

 \begin{coro}
 For every  saturated $C^2_\omega$-structure valued model ${\mathbf{V}}^{\langle {\bf A}, N\rangle}$ is a model of ZFC.
\end{coro}

\subsection{Paraconsistent models}

For every saturated $C_\omega$-structure valued model ${\mathbf{V}}^{\langle {\bf A}, N\rangle}$, let us consider  $u\in{\mathbf{V}}^{\langle {\bf A}, N \rangle}$ and every formula $\phi(x)$, we can always take   $1=\ter{ \neg \psi(u)}^{\langle {\bf A}, N \rangle}\in N_{\ter{ \psi(u)}^{\langle {\bf A}, N \rangle}}$ (recall that $1\in  N_{\ter{ \psi(u)}^{\langle {\bf A}, N \rangle}}$) and  $\ter{ \neg\neg \psi(u)}^{\langle {\bf A}, N \rangle}=  \ter{ \psi(u)}^{\langle {\bf A}, N \rangle}$ where $ \psi(x)$ is not the form $\neg\phi(x)$ for some formula $\phi(x)$.
In this case, we are taking a unique interpretation of $\ter{ \cdot}^{\langle {\bf A}, N \rangle}$; i.e., it is a mapping.  So, it is possible to see that Leibniz law is held and the axioms {\bf (C1)} and {\bf (C2)} (also {\bf (C3)})  are sound. These conditions work perfectly for induction over the  complexity of the formula; and more interestingly, ${\mathbf{V}}^{\langle {\bf A}, N \rangle}$ is paraconsistent with this interpretation. In order to see it, let us by  considering the following {\em relation of consequence}, we write: 

\begin{center}

 $\Gamma  \vDash_{{\mathbf{V}}^{\langle {\bf A}, N \rangle}} \varphi$, if $\termvalue{\gamma}^{\langle {\bf A}, N \rangle}=1$ for every $\gamma\in \Gamma$,\\ then  $\termvalue{\varphi}^{\langle {\bf A}, N \rangle}=1$ for arbitrary set of sentences $\Gamma \cup \{\varphi\}$. 

\end{center}

Now, take the formula   $\phi(x):= x\approx x$ and for any  $u\in {\mathbf{V}}^{\langle {\bf A}, N \rangle}$, let us permit to take $\ter{ \neg (u\approx u )}^{\langle {\bf A}, N \rangle}=  \ter{ u\approx u }^{\langle {\bf A}, N \rangle}=1$.  Consider now the  following saturated $\Cw$-structure ${\cal N}_3=\langle \{0,\frac{1}{2}, 1\},\wedge,\vee,\to, \neg 1, N_0, N_\frac{1}{2}, N_1\rangle$ where $N_0=\{1\}$, $N_\frac{1}{2}=\{1\}$ and $N_1=\{0,\frac{1}{2}, 1\}$. Suppose that  $v=\{\langle w, 1\rangle\}$, $u=\{\langle w, \frac{1}{2}\rangle\}$, then we have that  $\ter{u\approx v}^{{\cal N}_3}= \frac{1}{2}$. Therefore, $\{(u\approx u) ,\neg (u\approx u ) \}\not \vDash_{{\cal N}_3} (u\approx v)$.

Here, we must to pay attention that this new notion of {\em semantic consequence} is different from what was considered above in the text, but clearly they share the same logical theorems. The systems could be very different in some ways. To see that,  in the paper \cite{AFO3}, the authors have presented a family the logics $L^{\leq}_k$ and $L^1_k$ ($k$ is an integer number)  where   $L^{\leq}_k$  is built with degree-preserving construction and $L^1_k$ is a $1$-assertional both have the same algebraic counterpart. While $L^{\leq}_k$  is paraconsistent and non-algebraizable, $L^1_k$ is algebraizable and non-paraconsistent but they share the same theorems.

\section{Conclusions and future works}

The problem of independence of Continuum Hypothesis (HC) from $ZF_{C_\omega}$  can be treated using Boolean-valued models. It is clear that  Boolean-valued models validate the axioms of $C_\omega$. Therefore, the classical proof could be used to see that (HC) is independent of  $ZF_{C_\omega}$. The paraconsistent view of these affairs  will be part of our future research. Moreover, we are interested in developing a technique to obtain a  {\em paraconsistent forcing}. In particular, we will focus on paraconsistent set theories where Boolean-valued models are not models of these theories.

 On the other hand, in this note we have presented some paraconsistent set theories where the models validate {\em Choice Axiom} (CA). In this setting, we consider to be of high interest the problem of  independence in different PSTs.

 \subsection*{Acknowledgments}  
The author acknowledges   the support from São Paulo Research Foundation (FAPESP) through the {\em Jovem Pesquisador}   grant 2021/04883-0.

\end{document}